\documentclass[11pt]{article}
\usepackage[T1]{fontenc}
\usepackage{hyperref, amsfonts, amsmath, amssymb, amsthm, fullpage, enumitem, tikz, bm, mathtools, microtype, mleftright}
\usepackage[noabbrev, capitalize]{cleveref}

\definecolor{litegray}{RGB}{236,236,236}
\tikzstyle{vertex}=[circle, draw, fill=white, inner sep=0pt, minimum width=4pt]

\title{Classification of equiangular lines\\with fixed angle $\arccos(1/(1+2\sqrt2))$}
\author{
    Theodore Gossett\thanks{School of Mathematical and Statistical Sciences, Arizona State University, Tempe, AZ 85281, USA. Email: {\tt \{tgossett, zilinj, zwellner\}@asu.edu}.} \and Zilin Jiang\footnotemark[1]${\ }^{,}$\thanks{School of Computing and Augmented Intelligence, Arizona State University, Tempe, AZ 85281, USA. Supported in part by the Simons Foundation through its Travel Support for Mathematicians program and by U.S.\ taxpayers through NSF grant 2451581.} \and Adam Teets\thanks{Hamilton High School, Chandler, AZ 85248. Email: {\tt adamteets127@gmail.com}} \and Zoe Wellner\footnotemark[1]${\ }^{,}$
}
\date{}

\newtheorem{theorem}{Theorem}[section]

\newtheorem{lemma}[theorem]{Lemma}
\newtheorem{proposition}[theorem]{Proposition}
\newtheorem{problem}[theorem]{Problem}

\theoremstyle{definition}
\newtheorem{definition}[theorem]{Definition}

\theoremstyle{remark}
\newtheorem*{remark}{Remark}

\newcommand{\computer}{\tikz[scale=0.28, baseline=0.1ex, transform shape]{\draw[rounded corners=0.02] (0.15, 0.45) rectangle (0.85, 0.85) {};\draw[rounded corners=0.02] (0, 0.3) rectangle (1, 1) {};\draw[rounded corners=0.02] (0, 0) rectangle (1, 0.2) {};\draw (0.2,0.1) -- (0.3,0.1);\draw (0.5,0.1) -- (0.8,0.1);}}
\newenvironment{cproof}[1][Proof]{\begin{proof}[#1]}{\end{proof}}

\DeclareMathOperator{\rank}{rank}
\DeclarePairedDelimiter\abs{\lvert}{\rvert}
\DeclarePairedDelimiter\floor{\lfloor}{\rfloor}

\newcommand{\dset}[2]{\left\{{#1}\colon{#2}\right\}}
\newcommand{\sset}[1]{\left\{{#1}\right\}}

\newcommand{\C}{\mathcal{C}}
\newcommand{\G}{\mathcal{G}}
\newcommand{\N}{\mathbb{N}}
\newcommand{\R}{\mathbb{R}}

\newcommand{\rtwo}{\alpha^*}
\newcommand{\artwo}{\arccos(\rtwo)}
\newcommand{\cl}[1]{\langle #1 \rangle}
\newcommand{\gen}[1]{\langle #1 \rangle_\pm}

\newcommand{\cherry}{\tikz[baseline=-0.15ex]{\draw[darkgray, thick] (0,0)--(0.1,0.17320508075)--(0.2,0); \draw[darkgray, fill=white] (0,0) circle (0.05); \draw[darkgray, fill=white] (0.1,0.17320508075) circle (0.05); \draw[darkgray, fill=white] (0.2,0) circle (0.05);}}
\newcommand{\single}{\tikz[baseline=-0.5ex]{\draw[darkgray, thick] (0,0)--(0.2,0); \draw[darkgray, fill=white] (0,0) circle (0.05); \draw[darkgray, fill=white] (0.2,0) circle (0.05);}}
\newcommand{\defxy}[1]{\foreach \u/\x/\y in {#1} {\coordinate (\u) at (\x,\y);}}
\newcommand{\nodesvertex}[1]{\foreach \u in {#1} {\node[vertex] at (\u) {};}}
\newcommand{\drawedges}[1]{\foreach \u/\v in {#1} {\draw (\u) -- (\v);}}

\begin{document}

\maketitle

\begin{abstract}
	We determine the maximum number $N_\alpha(d)$ of equiangular lines with fixed angle $\arccos\alpha$ for $\alpha = 1/(1+2\sqrt2)$ in $d$-dimensional Euclidean space: $2,3,4,6,8,10,14,15,16,17,18,20,22$ for $d \in \{2,\dots,14\}$, and $\max(24, \lfloor 3(d-1)/2 \rfloor)$ for $d \ge 15$. This appears to be the first complete determination of $N_\alpha(d)$ in all dimensions $d$ for a fixed nontrivial $\alpha$, since the work of Lemmens and Seidel for $\alpha = 1/3$ in 1973.
\end{abstract}

\section{Introduction} \label{sec:intro}

A set of lines passing through the origin in $d$-dimensional Euclidean space $\R^d$ is called \emph{equiangular} if any two of them form the same angle. The problem of determining the maximum cardinality $N(d)$ of an equiangular line system in $\R^d$, formally posed by van Lint and Seidel \cite{vLS66} in 1966, is considered to be one of the foundational problems of algebraic graph theory. The exact value of $N(d)$ has been determined for only finitely many $d$. After decades of research, the state of the art for $d \le 43$ is summarized as follows.

\begin{center}
    \setlength{\tabcolsep}{4pt}
    \begin{tabular}{ccccccccccccccccc}
        $d$ & 2 & 3--4 & 5 & 6 & 7--14 & 15 & 16 & 17 & 18 & 19 & 20 & 21 & 22 & 23--41 & 42 & 43\\
        \hline
        $N(d)$ & 3 & 6 & 10 & 16 & 28 & 36 & 40 & 48 & 57--59 & 72--74 & 90--94 & 126 & 176 & 276 & 276--288 & 344
    \end{tabular}
\end{center}

This table has a long history. In 1948, Haantjes \cite{H48} determined $N(3) = 6$ and $N(4) = 6$ --- the optimal configuration in $\R^3$ is given by the $6$ diagonals of a regular icosahedron. Later, van Lint and Seidel \cite{vLS66} determined $N(d)$ for $d \in \sset{5,6,7}$. Lemmens and Seidel \cite{LS73} determined $N(d)$ for $d \in \sset{15,21,22,23}$, and they mentioned in their work Asche's and Taylor's construction implying that $N(19) \ge 72$ and $N(20) \ge 90$. Progress has been stalled for more than forty years. Greaves, Koolen, Munemasa, and Sz\"{o}ll\H{o}si \cite{GKMS16} improved two long-standing upper bounds on $N(14)$ and $N(16)$, and Barg and Yu \cite{BY14} determined $N(d) = 276$ for every $d \in \sset{24, \dots, 41}$ and $N(43) = 344$, while leaving open the problem of $N(42)$, with $276 \le N(42) \le 288$. Research on $N(d)$ has remained active in recent years: Greaves and Yatsyna \cite{GY19} improved the upper bound $N(17) \le 49$; Greaves, Syatriadi and Yatsyna in two subsequent works \cite{GSY21,GSY23} determined $N(14) = 28$, $N(16) = 40$, and $N(17) = 48$, and they also improved the bounds to $N(18) \ge 57$, $N(19) \le 74$ and $N(20) \le 94$; and most recently, Greaves and Syatriadi \cite{GS24} improved the upper bound $N(18) \le 59$.

The question of determining the maximum number $N_\alpha(d)$ of equiangular lines in $\R^d$ with common angle $\arccos\alpha$ was first raised by Lemmens and Seidel \cite{LS73} in 1973, who determined that $N_{1/3}(d)$ equals $2,4,6,10,16$ for $d \in \sset{2,\dots,6}$, and $\max(28, 2(d-1))$ for $d \ge 7$. They also conjectured that $N_{1/5}(d) = \max(276, \floor{3(d-1)/2})$ for $d \ge 23$, which was recently resolved by Cao, Koolen, Lin and Yu \cite{CKLY22}. When $2 \le d \le 22$, here are the known results.

\begin{center}
    \setlength{\tabcolsep}{4pt}
    \begin{tabular}{ccccccccccccccccccccccc}
        $d$ & 2 & 3 & 4 & 5 & 6 & 7 & 8 & 9 & 10 & 11 & 12 & 13 & 14 & 15 & 16 & 17 & 18 & 19 & 20 & 21 & 22 \\
        \hline
        $N_{1/5}(d)$ & 2 & 3 & 4 & 6 & 7 & 9 & 10 & 12 & 16 & 18 & 20 & 26 & 28 & 36 & 40 & 48 & 57--59 & 72--74 & 90--94 & 126 & 176
    \end{tabular}
\end{center}

In the above table, Greaves, Koolen, Munemasa, and Sz\"{o}ll\H{o}si \cite{GKMS16} determined $N_{1/5}(d)$ for $d \in \sset{2,\dots, 11} \cup \sset{13}$, and Sz\"{o}ll\H{o}si and \"{O}sterg\aa rd \cite{SO18} determined it for $d = 12$. It is not a coincidence that the known results on $N(d)$ and $N_{1/5}(d)$ agree for $d \in \sset{14, \dots, 41}$ --- it follows from \cite[Theorem 3.7]{LS73} asserting that $N(d) \le \max(N_{1/3}(d), N_{1/5}(d), \floor{48d/(49-d)})$ for $d \le 48$.

Recent work \cite{B16,BDKS18,JP20} culminated in a solution \cite{JTYZZ21} of Jiang, Tidor, Yao, Zhang and Zhao to the problem of determining $N_\alpha(d)$ in high dimensions. In particular, they proved for every integer $k \ge 2$ and 
$d$ sufficiently large that
\[
    N_{\rtwo}(d) = \floor{3(d-1)/2} \text{ and }N_{1/(2k-1)}(d) = \floor{kd/(k-1)}.
\]
Here and throughout, we fix the constant
\[
    \rtwo = \frac{1}{1 + 2\sqrt2}.
\]
We refer the reader to the excellent survey \cite{Z24} of Zhao for the precise statement of the relevant result of \cite{JTYZZ21} and the connection to graph eigenvalue multiplicity.

Given that $\rtwo$ sits between $1/3$ and $1/5$, one might hope to determine $N_{\rtwo}(d)$ for all dimensions $d$ before tackling $N_{1/5}(d)$. Indeed, we resolve this problem.

\begin{theorem} \label{thm:main}
    The maximum number of equiangular lines with fixed angle $\artwo$ in $d$-dimensional Euclidean space satisfies
    \[
        N_{\rtwo}(d) = \begin{cases}
            2,3,4,6,8,10,14,15,16,17,18,20,22 & \text{for }d \in \sset{2,\dots, 14}, \\
            \max(24, \floor{3(d-1)/2}) & \text{for }d \ge 15.
        \end{cases}
    \]
\end{theorem}

We prove the main theorem in the strongest possible sense --- we classify equiangular line systems with fixed angle $\artwo$ in $\R^d$, for all $d$, up to orthogonal transformations. To state our classification theorem, we recall the connection between equiangular lines, graphs and their Seidel matrices. The \emph{Seidel matrix} of a graph $G$, denoted $S(G)$, is defined as $J - I - 2A(G)$, where $A(G)$ is the adjacency matrix of $G$, and the \emph{smallest Seidel eigenvalue} of $G$ is the smallest eigenvalue of $S(G)$.

It was observed by van Lint and Seidel~\cite{vLS66} that classifying equiangular line systems with fixed angle $\arccos\alpha$ up to orthogonal transformations is equivalent to classifying graphs with smallest Seidel eigenvalue at least $-1/\alpha$ up to \emph{switching equivalence} (cf. \cite[Chapter~11]{GR01}).

\begin{definition}[Switching equivalence]
    Given a graph $G$ and a subset $U \subseteq V(G)$, the \emph{switching} of $G$ with respect to $U$, denoted $G^U$, is the graph obtained from $G$ by toggling adjacency between every vertex in $U$ and every vertex in $V(G) \setminus U$. Two graphs $G$ and $H$ are \emph{switching equivalent} if there exists $U \subseteq V(G)$ such that $H$ is isomorphic to $G^U$.
\end{definition}

To describe most of the graphs with smallest Seidel eigenvalue at least $-1/\rtwo$, we introduce the following notion. All subgraphs are induced throughout this paper.

\begin{definition}[Switching closure]
    Given a family $\G$ of graphs, the \emph{switching closure}, denoted $\gen{\G}$, is the smallest class of graphs containing every member of $\G$ that is closed under taking subgraphs, vertex-disjoint union, and switching equivalence.
\end{definition}

In \cref{sec:idea}, we note the simple fact that every graph in $\gen\cherry$ has smallest Seidel eigenvalue at least $-1/\rtwo$, where $\cherry$ denotes the cherry graph (the path of length $3$) throughout. We enumerate all other graphs with smallest Seidel eigenvalue at least $-1/\rtwo$, and summarize the results as follows.

\begin{theorem} \label{thm:stat}
    Up to switching equivalence, there are a total of 63673 graphs outside $\gen\cherry$ with smallest Seidel eigenvalue at least $-1/\rtwo$ with the following statistics, where the rank refers to the rank of $S(G) + I / \rtwo$ for a graph $G$.
    \begin{flushleft}
        \begin{tabular}{cccccccccccc}
            $d$ & 5 & 6 & 7 & 8 & 9 & 10 & 11 & 12 & 13 & 14\\
            \hline
            \# & 1 & 6 & 25 & 82 & 294 & 1047 & 3533 & 8929 & 14921 & 15799\\
            \hline
            minimum rank & 5 & 6 & 6 & 6 & 7 & 7 & 8 & 8 & 8 & 8
        \end{tabular}
    \end{flushleft}
    \begin{flushright}
        \begin{tabular}{cccccccccccccc}
            15 & 16 & 17 & 18 & 19 & 20 & 21 & 22 & 23 & 24 & 25 & 26 & 27 & 28\\
            \hline
            10555 & 5030 & 2034 & 831 & 325 & 140 & 63 & 31 & 13 & 6 & 3 & 2 & 2 & 1\\
            \hline
            9 & 10 & 11 & 12 & 13 & 13 & 14 & 14 & 15 & 15 & 25 & 26 & 27 & 28
        \end{tabular}
    \end{flushright}
    Moreover, up to switching equivalence, the graph achieving the minimum rank is unique for $n \in \sset{5,7,8,13,14,22,24,28}$, and when $n \in \sset{5,8,14,28}$, the unique graph is switching equivalent to the cycle graph $C_5$, the cycle graph $C_8$, the Heawood graph (see \cref{fig:heawood}) and the complement of the line graph of the complete graph $K_8$, respectively.
\end{theorem}

\begin{figure}
    \centering
    \begin{tikzpicture}[scale=1.5, thick]
        \foreach \i in {0,...,13} {
            \node[vertex] (\i) at ({360/14 * \i}:1) {};
        }
        \foreach \i in {0,...,13} {
            \pgfmathtruncatemacro{\next}{Mod(\i + 1, 14)}
            \draw (\i) -- (\next);
        }
        \foreach \i in {0,...,13} {
            \pgfmathtruncatemacro{\iseven}{Mod(\i, 2)}
            \ifnum\iseven=0
                \pgfmathtruncatemacro{\target}{Mod(\i + 5, 14)}
            \else
                \pgfmathtruncatemacro{\target}{Mod(\i - 5, 14)}
            \fi
            \draw (\i) -- (\target);
        }
    \end{tikzpicture}
    \caption{Heawood graph.}
    \label{fig:heawood}
\end{figure}
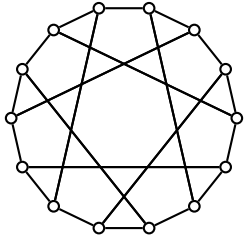

The rest of the paper is organized as follows. In \cref{sec:idea}, we prove the main theorem assuming \cref{thm:stat}. In \cref{sec:forb}, we identify five graphs such that every graph outside $\gen\cherry$ contains at least one of them as a subgraph. In \cref{sec:enum}, we prove \cref{thm:stat} by extending these five graphs. We conclude with a discussion of equiangular lines with fixed angle $\arccos(1/3)$ and an open problem in \cref{sec:open}.

Parts of our proofs are computer-assisted, using validated arithmetic. We describe each such proof in detail so that it can be independently reproduced. To help readers identify computer-assisted proofs, we use a bespoke symbol \computer\ at the conclusion of each such proof. All our code, written in C++, is available as ancillary files in the arXiv version of this paper. We deliberately avoid third-party libraries, making it straightforward to adapt the code to other programming languages.

\section{Proof of main theorem} \label{sec:idea}

To establish the main theorem \cref{thm:main} assuming \cref{thm:stat}, we first formally state the graph-theoretic characterization of equiangular line systems.

\begin{theorem}[van Lint and Seidel \cite{vLS66}] \label{thm:correspondence}
    There is a one-to-one correspondence between systems of $n$ equiangular lines with fixed angle $\alpha$ in $\R^d$ modulo orthogonal transformations and graphs of order $n$ with smallest Seidel eigenvalue at least $-1/\alpha$ and the multiplicity of $-1/\alpha$ as a Seidel eigenvalue at least $n - d$ modulo switching equivalence. \qed
\end{theorem}

Under the correspondence, there are two kinds of equiangular line systems: those that arise from graphs in $\gen\cherry$ and those from graphs not in $\gen\cherry$. The next proposition deals with the former case, whereas \cref{thm:stat} deals with the latter case. To this end, we need the following notion.

\begin{definition}[Closure]
    Given a family $\G$ of graphs, the \emph{closure}, denoted $\cl\G$, of $\G$ is the smallest class of graphs containing $\G$ that is closed under taking subgraphs and vertex-disjoint union.
\end{definition}

\begin{proposition} \label{lem:cherry-closure}
    Every graph in $\gen\cherry$ has smallest Seidel eigenvalue at least $-1/\rtwo$. Moreover, for every $d \in \N^+$, the maximum order of a graph $G$ in $\gen\cherry$, for which the rank of $S(G) + I/\rtwo$ is at most $d$, is exactly $\max(d, \floor{3(d-1)/2})$.
\end{proposition}

The proof closely follows the proof of Theorem 1.2 and Proposition 3.2 in \cite{JTYZZ21}.

\begin{proof}
    Let $G$ be a graph in $\gen\cherry$. Since switching equivalence preserves the eigenvalues of Seidel matrices, without loss of generality, we may assume that $G$ is in $\cl\cherry$, that is, each connected component of $G$ is an isolated vertex, a single edge, or a cherry. Since the largest eigenvalue of $\cherry$ is $\sqrt{2}$, every connected component of $G$ has largest eigenvalue at most $\sqrt{2}$, and so does $G$ itself. Thus
    \begin{equation} \label{eqn:cherry-seidel}
        S(G) + I/\rtwo = J - I - 2A(G) + I/\rtwo = J + 2\mleft(\sqrt{2}I - A(G)\mright),
    \end{equation}
    which is positive semidefinite. Therefore, the smallest Seidel eigenvalue of $G$ is at least $-1/\rtwo$.
    
    Now, suppose further that the rank of $S(G) + I/\rtwo$ is at most $d$. Let $c$ be the number of cherries in $G$.

    If $c = 0$, then the largest eigenvalue of $G$ is at most $1$, and so according to \eqref{eqn:cherry-seidel}, the matrix $S(G) + I/\rtwo$ is positive definite, which implies that the rank of $S(G) + I/\rtwo$ is exactly $d$.

    Finally, we consider the case where $c > 0$. Since both $J$ and $\sqrt2 I - A(G)$ are positive semidefinite,
    \[
        \ker\mleft(J + 2\mleft(\sqrt2 I - A(G)\mright)\mright) = \ker(J) \cap \ker\mleft(\sqrt2 I - A(G)\mright).
    \]
    By the Perron--Frobenius theorem, there is an eigenvector of $A(G)$ associated with $\sqrt2$ that has nonnegative entries. This vector lies in $\ker(\sqrt2 I - A(G))$, but not in $\ker(J)$, implying that $\dim\ker\mleft(J + 2\mleft(\sqrt2 I - A(G)\mright)\mright) \le \dim\ker(\sqrt2 I - A(G)) - 1$. By the rank-nullity theorem, we obtain
    \[
        n - c = \rank(\sqrt2 I - A(G)) \le \rank(J + 2\mleft(\sqrt2 I - A(G)\mright)) \le d - 1.
    \]
    Since $c \le n/3$, we have $n \le \floor{3(d-1)/2}$. Moreover, when $G$ is the vertex-disjoint union of $\floor{(d-1)/2}$ cherries along with $(d-1) - 2\floor{(d-1)/2}$ isolated vertices, the order of $G$ is exactly $\floor{3(d-1)/2}$, and the rank of $\sqrt2 I - A(G)$ is exactly $d-1$, which implies that the rank of $S(G) + I/\rtwo$ is at most $d$ via \eqref{eqn:cherry-seidel}.
\end{proof}

\begin{proof}[Proof of \cref{thm:main} assuming \cref{thm:stat}]
    By \cref{thm:correspondence}, for each dimension $d$, we seek the largest order $N_{\rtwo}(d)$ of a graph $G$ with smallest Seidel eigenvalue at least $-1/\rtwo$ such that the rank of $S(G) + I/\rtwo$ is at most $d$. 
    
    If $G$ is a graph in $\gen\cherry$, then by \cref{lem:cherry-closure}, the largest order of $G$ is exactly $\max(d, \floor{3(d-1)/2})$.

    Suppose that $G$ is not in $\gen\cherry$. According to \cref{thm:stat}, the largest order of such a graph $G$ with rank of $S(G) + I/\rtwo$ at most $d$ is listed as follows.

    \begin{center}
        \begin{tabular}{ccccccccccccccccc}
            $d$ & 5 & 6 & 7 & 8 & 9 & 10 & 11 & 12 & 13 & 14 & 15--24 & 25 & 26 & 27 & 28--$\infty$\\
            \hline
            maximum order & 5 & 8 & 10 & 14 & 15 & 16 & 17 & 18 & 20 & 22 & 24 & 25 & 26 & 27 & 28
        \end{tabular}
    \end{center}

    Taking the maximum of the two cases, we obtain the values of $N_{\rtwo}(d)$ for all $d \in \N^+$.
\end{proof}

\section{Minimal forbidden subgraphs} \label{sec:forb}

To identify graphs such that every graph outside $\gen\cherry$ contains at least one of them as a subgraph, we take the most economical approach.

\begin{definition}[Minimal forbidden subgraph]
    Given a class $\C$ of graphs that is closed under taking subgraphs, a graph $F$ is a \emph{minimal forbidden subgraph} for $\C$ if $F$ itself is not in $\C$, but every proper subgraph of $F$ is in $\C$.
\end{definition}

To enumerate the minimal forbidden subgraphs for $\gen\cherry$, we show that a large graph in $\gen\cherry$ is uniquely represented by a graph in $\cl\cherry$ under switching equivalence, and then characterize the minimal forbidden subgraphs for $\cl\cherry$.

\begin{proposition} \label{lem:uniq}
    For every $G \in \cl\cherry$ of order at least $7$, and every $U \subseteq V(G)$, if the switching $G^U$ of $G$ with respect to $U$ is also in $\cl\cherry$, then $U$ is trivial, that is, $U = \varnothing$ or $U = V(G)$.
\end{proposition}

\begin{proof}
    Suppose that $G$ is a graph in $\cl\cherry$ of order $n \ge 7$, and $U \subseteq V(G)$ such that $G^U$ is also in $\cl\cherry$. We may assume that $\abs{U} \le n/2$ by taking the complement of $U$ if necessary. Assume for the sake of contradiction that $U \neq \varnothing$. Set $W = V(G) \setminus U$. Clearly $\abs{W} \ge 4$. For every vertex $u \in U$, it has at least $\abs{W \setminus N_G(u)}$ neighbors inside $W$ in $G^U$, where $N_G(u)$ denotes the set of neighbors of $u$ in $G$. Since $G^U \in \cl\cherry$, its maximum degree is at most $2$, and so $\abs{W \setminus N_G(u)} \le 2$. It must be the case that $\abs{W} = 4$, $\abs{N_G(u)} = 2$, and $N_G(u) \subseteq W$ for every $u \in U$. Since $G[\sset{u} \cup N_G(u)]$ has to be a connected component of $G$ that is isomorphic to $\cherry$ for every $u \in U$, we have $2\abs{U} \le \abs{W} = 4$, which contradicts $n \ge 7$.
\end{proof}

\begin{proposition} \label{lem:min-forb}
    Any minimal forbidden subgraph for $\cl\cherry$ has at most $4$ vertices.
\end{proposition}

\begin{proof}
    Suppose that $G$ is a minimal forbidden subgraph for $\cl\cherry$. Since $G \notin \cl\cherry$ but every proper subgraph of $G$ is in $\cl\cherry$, it must be the case that $G$ is connected. There exists a vertex $v$ of $G$ such that $G - v$ is still connected. Since $G-v$ is a subgraph of $\cherry$, the graph $G$ has at most $4$ vertices.
\end{proof}

We bound the order of minimal forbidden subgraphs for $\gen\cherry$ in preparation for their enumeration.

\begin{lemma} \label{lem:min-forb-switching}
    Any minimal forbidden subgraph for $\gen\cherry$ has at most $8$ vertices.
\end{lemma}

\begin{proof}
    Let $G$ be a minimal forbidden subgraph for $\gen\cherry$ on $n$ vertices. Assume for the sake of contradiction that $n \ge 9$. Fix a vertex $v$ of $G$. Minimality means $G - v$ is in $\gen\cherry$. By applying a switching to $G$, we may assume that $G - v$ is in $\cl\cherry$, and that $v$ has degree at most $n/2$.
    
    We claim that $G - u$ is in $\cl\cherry$ for every $u \in V(G)$. Take an arbitrary $u \in V(G)$. We may assume that $u \neq v$. Since $G - u \in \gen\cherry$, there exists $U \subseteq V(G)\setminus \sset{u}$ such that $(G-u)^U \in \cl\cherry$. We may assume that $v \not\in U$ by taking the complement of $U$ as a subset of $V(G) \setminus \sset{u}$ in case $v \in U$. Since $G - v \in \cl\cherry$, clearly $G - \sset{u,v}$ is in $\cl\cherry$. Since $(G - \sset{u,v})^{U}$ is also in $\cl\cherry$, by \cref{lem:uniq}, we have that $U$ is trivial as a subset of $V(G) \setminus \sset{u,v}$. Suppose for a moment that $U = V(G) \setminus \sset{u,v}$. Since the degree of $v$ in $(G - u)^{\sset{v}}$ is at least $n / 2 - 1$, the graph $(G - u)^{\sset{v}} \not\in \cl\cherry$, which contradicts the fact that $(G - u)^{\sset{v}} = (G - u)^U \in \cl\cherry$. Therefore $U = \varnothing$, and $G - u = (G - u)^U \in \cl\cherry$.
    
    Based on the claim, we observe that $G$ is a minimal forbidden subgraph for $\cl\cherry$ on at least $9$ vertices, which contradicts \cref{lem:min-forb}.
\end{proof}

Our enumeration of minimal forbidden subgraphs for $\gen\cherry$ is a straightforward brute force approach.

\begin{lemma}
    Up to switching equivalence, there are $5$ minimal forbidden subgraphs for $\gen\cherry$ listed in \cref{fig:min-forb}.
\end{lemma}

\begin{figure}
    \centering
    \begin{tikzpicture}[scale=.5, thick]
        \fill[litegray, rounded corners=12pt] (-1,-1) rectangle (5,1);
        \defxy{a/0/0,b/1/0,c/2/0,d/3/0,e/4/0}
        \drawedges{a/b,b/c,c/d}
        \nodesvertex{a,b,c,d,e}
        \begin{scope}[shift={(7.8660254,0)}]
            \fill[litegray, rounded corners=12pt] (-1.8660254,-1) rectangle (3,1);
            \defxy{a/-.8660254/.5,b/-.8660254/-.5,c/0/0,d/1/0,e/2/0}
            \drawedges{a/b,b/c,c/a,d/e}
            \nodesvertex{a,b,c,d,e}
        \end{scope}
        \begin{scope}[shift={(13.8660254,0)}]
            \fill[litegray, rounded corners=12pt] (-1.8660254,-1) rectangle (4,1);
            \defxy{a/-.8660254/.5,b/-.8660254/-.5,c/0/0,d/1/0,e/2/0,f/3/0}
            \drawedges{b/c,a/c,c/d,e/f}
            \nodesvertex{a,b,c,d,e,f}
        \end{scope}
        \begin{scope}[shift={(20.8660254,0)}]
            \fill[litegray, rounded corners=12pt] (-1.8660254,-1) rectangle (2.8660254,1);
            \defxy{a/-.8660254/.5,b/-.8660254/-.5,c/0/0,d/1/0,e/1.8660254/.5,f/1.8660254/-.5}
            \drawedges{b/c,a/c,c/d,d/f,d/e}
            \nodesvertex{a,b,c,d,e,f}
        \end{scope}
        \begin{scope}[shift={(26.5980762,0)}]
            \fill[litegray, rounded corners=12pt] (-1.8660254,-1) rectangle (5,1);
            \defxy{a/-.8660254/.5,b/-.8660254/-.5,c/0/0,d/1/0,e/2/0,f/3/0,g/4/0}
            \drawedges{b/c,c/a,c/d}
            \nodesvertex{a,b,c,d,e,f,g}
        \end{scope}
    \end{tikzpicture}
    \caption{Minimal forbidden subgraphs for the switching closure of the cherry graph.}
    \label{fig:min-forb}
\end{figure}
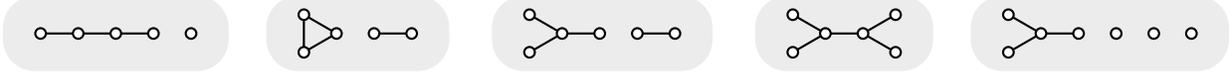

\begin{cproof}
    By \cref{lem:min-forb-switching}, the order of any minimal forbidden subgraph for $\gen\cherry$ is at most $8$. To enumerate minimal forbidden subgraphs $G$ for $\gen\cherry$ up to switching equivalence, we may assume that $G$ on $n$ vertices can be obtained from a graph in $\cl\cherry$ on $n-1$ vertices by adding a new vertex and specifying its neighbors among the $n-1$ existing vertices.
    
    Our algorithm first stores all graphs in $\cl\cherry$ on up to $8$ vertices in \texttt{base} up to switching equivalence, and then for graphs in \texttt{base} on up to $7$ vertices, we add a new vertex, exhaust all possible vertex subsets as its neighborhood, and check whether each new graph $G$ is a minimal forbidden subgraph for $\gen\cherry$ --- $G$ should not be switching equivalent to any of the graphs in \texttt{base}, but $G - v$ should be switching equivalent to some graph in \texttt{base} for every vertex $v$ of $G$. Finally, we add each such minimal forbidden subgraph to \texttt{dict}, provided it is not switching equivalent to any existing member of \texttt{dict}. Our code is available as \texttt{min\_forb.cpp}, an auxiliary file in the arXiv version of the paper.
\end{cproof}

For the rest of this section, we provide further details of our implementation.

\subsection{Hash function of graphs and Seidel degrees} \label{subsec:hash}

A common routine in the above algorithm is to decide whether a graph $G$ is switching equivalent to any graph in a given set of graphs. A brute force implementation would still run within a reasonable time limit. Nevertheless, we carry out the more efficient implementation since the code will be reused in \cref{sec:enum}.

Suppose that we need to check whether a graph $G$ is switching equivalent to any existing member of \texttt{set}. To efficiently detect switching equivalence, we maintain a hash table \texttt{hash} of existing members of \texttt{set} using the following hash function.

\begin{definition}
    Given a graph $G$, its vertex $v$, and $k \in \N$, the $k$-th Seidel degree of $v$ in $G$ is the $(v,v)$-entry of the $k$-th power of the Seidel matrix of $G$.
\end{definition}

\begin{proposition} \label{lem:seidel-deg}
    For every graph $G$, every vertex subset $U$ of $G$, every vertex $v$ of $G$, and every $k \in \N$, the $k$-th Seidel degree of $v$ in $G$ is equal to that of $v$ in $G^U$.
\end{proposition}

\begin{proof}
    Define $D(U)$ to be the diagonal matrix on $V(G)$ by $D_{v,v} = 1$ if $v \in U$ and $D_{v,v} = -1$ otherwise. The Seidel matrices of $G$ and $G^U$ are related by $S(G^U) = D(U)S(G)D(U)$. Finally, since $D(U)^2 = I$, we have $S(G^U)^k = D(U)S(G)^kD(U)$, whose diagonal entries agree with those in $S(G)^k$.
\end{proof}

For a graph $G$ of order $n$, its hash value is the sorted sequence of third Seidel degrees. We skip first and second Seidel degrees because they are always $0$ and $n-1$, respectively. Clearly, when two graphs are switching equivalent, their hash values are equal by \cref{lem:seidel-deg}.

This allows us to test whether $G$ is switching equivalent to any existing member of \texttt{set} by examining only the members of \texttt{hash[hv]}, where \texttt{hv} is the hash value of $G$. In addition, when we search for a vertex subset $U$ of $G$ and an isomorphism between $G^U$ and a member $G'$ of \texttt{hash[hv]}, we only map a vertex of $G$ to a vertex of $G'$ that has the same third Seidel degree.

\section{Enumeration of exceptional graphs} \label{sec:enum}

In this section, we prove \cref{thm:stat} by extending the minimal forbidden subgraphs for $\gen\cherry$ in \cref{fig:min-forb} to graphs with smallest Seidel eigenvalue at least $-1/\rtwo$.

\begin{cproof}[Proof of \cref{thm:stat}]
    To enumerate the graphs $G \notin \gen\cherry$ with smallest Seidel eigenvalue at least $-1/\rtwo$ up to switching equivalence, we generate these graphs in \texttt{dict} based on their order. At the start, \texttt{dict[5]}, \texttt{dict[6]}, \texttt{dict[7]} consist of those minimal forbidden subgraphs for $\gen\cherry$ in \cref{fig:min-forb} on $5, 6, 7$ vertices whose smallest Seidel eigenvalue is at least $-1/\rtwo$, and the counter \texttt{n} is set to \texttt{5}. It turns out that the smallest Seidel eigenvalue of the second graph in \cref{fig:min-forb} is less than $-1/\rtwo$, and so we skip it.

    Whenever \texttt{dict[n]} is nonempty, we iterate through members of \texttt{dict[n]} before incrementing the counter \texttt{n} by $1$. For each $G$ in \texttt{dict[n]}, we carry out the following steps.

    \begin{enumerate}[label=(\roman*)]
        \item We compute $\rank(S(G) + I/\rtwo)$ and record the current minimum rank in \texttt{min\_rank}.
        \item We connect a new vertex to a vertex subset $U$ of $G$ in every possible way to obtain new graphs $H$. See \cref{subsec:foresight} for a more efficient implementation.
        \item We store those graphs $H$ that satisfy $S(H) + I/\rtwo \succeq 0$ in a temporary array \texttt{candidates}.
        \item We merge \texttt{candidates} into \texttt{dict[n+1]} (cf.~\cref{subsec:hash}).
    \end{enumerate}

    Our program terminates at $n = 29$ because \texttt{dict[29]} is empty. We also output the graphs achieving the minimum rank for each $n \in \sset{5, \dots, 28}$, and write all the graphs in \texttt{dict[n]} together with their rank to the file \texttt{cherry-<n>.txt}. Our code is available as the ancillary file \texttt{enum.cpp} in the arXiv version of the paper.

    Based on the output, we observe that the graph achieving the minimum rank is unique for $n \in \sset{5,7,8,13,14,22,24,28}$. When $n \in \sset{5,8,14,28}$, the unique graph is switching equivalent to the cycle graph $C_5$, the cycle graph $C_8$, the Heawood graph and the complement of the line graph of the complete graph $K_8$, respectively, whereas when $n \in \sset{7,13,22,24}$, the unique graph is not switching equivalent to any regular graph.
\end{cproof}

\begin{remark}
    Although no regular graph is switching equivalent to the unique graph (with smallest Seidel eigenvalue at least $-1/\rtwo$) achieving the minimum rank for each $n \in \sset{7, 13, 22, 24}$, one can still seek an easy-to-describe switching of these graphs. For example, when $n = 24$, the unique graph is switching equivalent to the following: take the Heawood graph in \cref{fig:heawood} with bipartition $A$ and $B$; add four vertices $c_1, c_2, c_3, c_4$, each connected to all vertices in $B$; then add six more vertices, each connected to a distinct pair of vertices in $\sset{c_1, c_2, c_3, c_4}$.
\end{remark}

\begin{remark}
    On a Mac mini equipped with an Apple M4 chip and 16 GB of memory, the program terminated in under 3 minutes.
\end{remark}

\subsection{Adding a new vertex} \label{subsec:foresight}

To accelerate the generation of graphs, we leverage the computation done for $G$ in \texttt{dict[n]} based on the following observation. Suppose that $G$ is a graph such that $S(G) + I/\rtwo \succeq 0$. For $U \subseteq V(G)$, denote by $G_U^+$ the graph obtained from $G$ by connecting a new vertex to vertices in $U$. Set
\[\mathcal{U} = \dset{U \subseteq V(G)}{ S(G_U^+) + I/\rtwo \succeq 0}.\]
For every $W \subseteq V(H)$, where $H = G_U^+$ for some $U \subseteq V(G)$, note that $S(H^+_{W}) + I/\rtwo \succeq 0$ implies that $S(G^+_{W \cap V(G)}) + I/\rtwo \succeq 0$, which is equivalent to $W \cap V(G) \in \mathcal{U}$. In other words, when we connect a new vertex to a vertex subset $W$ of $H$, we need only consider $W$ satisfying $W \cap V(G) \in \mathcal{U}$.

To keep track of the set $\mathcal{U}$ defined for $G$, we add an attribute \texttt{possible\_subsets}, a list of distinct vertex subsets of $G$, to each graph $G$. For the minimal forbidden subgraphs for $\gen\cherry$ in \cref{fig:min-forb}, its \texttt{possible\_subsets} consists of all the nonempty vertex subsets. For each $G$ in \texttt{dict[n]}, instead of connecting a new vertex, say $v$, to a vertex subset $U$ of $G$ in every possible way, we connect $v$ to $S$ for every $S$ in \texttt{possible\_subsets} of $G$. To obtain \texttt{possible\_subsets} of the graphs obtained from $G$, we initially set \texttt{new\_possible\_subsets} to be the empty list. For each $U$ in \texttt{possible\_subsets} of $G$, check whether $S(G_U^+) + I/\rtwo \succeq 0$, and if so, we append both $U$ and $U \cup \sset{v}$ to \texttt{new\_possible\_subsets}, and we append $G_U^+$ to \texttt{candidates}. Finally, we set \texttt{possible\_subsets} of each graph $G'$ in \texttt{candidates} to \texttt{new\_possible\_subsets}.

\begin{remark}
    The same idea of carrying computation forward previously appeared in \cite{AJ25}, where Acharya and Jiang classified connected graphs $G$ satisfying $A(G) + \lambda^* I \succeq 0$, where $\lambda^* \approx 2.01980$.
\end{remark}

\subsection{Positive semidefiniteness and rank}

Testing whether $S(G) + I/\rtwo \succeq 0$ amounts to verifying that all the roots of $p_G(x - 1/\rtwo)$ are non-negative, where $p_G$ is the characteristic polynomial of $S(G)$. We recall the following criterion, which is a special case of Descartes' rule of signs: for $x_1, \dots, x_n \in \R$, they are non-negative if and only if their elementary symmetric functions are all non-negative. Suppose that
\[
    p_G\mleft(x - 1/\rtwo\mright) = c_nx^n + \dots + c_0.
\]
We can conclude that $S(G) + I/\rtwo \succeq 0$ if and only if $(-1)^{n-k}c_k \ge 0$ for every $k \in \sset{0, \dots, n}$.

To compute the characteristic polynomial $p_G$ of $S(G)$, we use the Faddeev--LeVerrier algorithm (see \cite[Chapter IV Section 5]{G98}), throughout which only integers are involved. Computing the coefficients of $p_G(x - 1/\rtwo)$ can be done entirely within the integral domain $\mathbb{Z}[\sqrt2]$.

As a byproduct, the rank of $S(G) + I/\rtwo$ can be computed as $n - m$, where $m$ is the multiplicity of $0$ as a root of $p_G(x - 1/\rtwo)$, that is, $m$ is the index such that $c_m \neq 0$ but $c_i = 0$ for all $i < m$.

\begin{remark}
    One caveat in the implementation comes from the possibility that the integers involved in the Faddeev--LeVerrier algorithm as well as those in the coefficients of $p_G(x)$ and $p_G(x - 1/\rtwo)$ can be quite large theoretically. Fortunately, in practice, the integers involved in the computation never exceed $3.42 \times 10^{21}$ in absolute value, and so signed 128-bit integers are sufficient for the computation mentioned in this subsection.
\end{remark}

\section{Concluding remarks} \label{sec:open}

\subsection{Equiangular lines with fixed angle \texorpdfstring{$\arccos(1/3)$}{arccos(1/3)}}

One can adapt the whole process of this paper to equiangular lines with fixed angle $\arccos(1/3)$. Since that angle is well-understood, the consistency of our enumeration with known results gives further confidence in the implementation of the computer-assisted proofs. We apply the adaptation as follows. Notice that every graph in $\gen\single$ has smallest Seidel eigenvalue at least $-3$. To enumerate graphs with smallest Seidel eigenvalue at least $-3$ outside $\gen\single$, we list the minimal forbidden subgraphs for $\gen\single$ and then extend them to graphs with smallest Seidel eigenvalue at least $-3$. We omit the details since the procedure mirrors that of \cref{sec:forb} and \cref{sec:enum}.

\begin{lemma}
    Up to switching equivalence, there are $4$ minimal forbidden subgraphs for $\gen\single$ listed in \cref{fig:min-forb-single}. \hfill\computer
\end{lemma}

\begin{figure}
    \centering
    \begin{tikzpicture}[scale=.5, thick]
        \fill[litegray, rounded corners=12pt] (-1,-1) rectangle (5,1);
        \defxy{a/0/0,b/1/0,c/2/0,d/3/0,e/4/0}
        \drawedges{a/b,b/c}
        \nodesvertex{a,b,c,d,e}
        \begin{scope}[shift={(7,0)}]
            \fill[litegray, rounded corners=12pt] (-1,-1) rectangle (5,1);
            \defxy{a/0/0,b/1/0,c/2/0,d/3/0,e/4/0}
            \drawedges{a/b,b/c,d/e}
            \nodesvertex{a,b,c,d,e}
        \end{scope}
        \begin{scope}[shift={(14,0)}]
            \fill[litegray, rounded corners=12pt] (-1,-1) rectangle (5,1);
            \defxy{a/0/0,b/1/0,c/2/0,d/3/0,e/4/0}
            \drawedges{a/b,b/c,c/d}
            \nodesvertex{a,b,c,d,e}
        \end{scope}
        \begin{scope}[shift={(21.8660254,0)}]
            \fill[litegray, rounded corners=12pt] (-1.8660254,-1) rectangle (3,1);
            \defxy{a/-.8660254/.5,b/-.8660254/-.5,c/0/0,d/1/0,e/2/0}
            \drawedges{a/b,b/c,c/a,d/e}
            \nodesvertex{a,b,c,d,e}
        \end{scope}
    \end{tikzpicture}
    \caption{Minimal forbidden subgraphs for the switching closure of the single-edge graph.}
    \label{fig:min-forb-single}
\end{figure}

We obtain the following classification of equiangular line systems with fixed angle $\arccos{1/3}$ up to orthogonal transformations.

\begin{theorem} \label{thm:main-1}
    Up to switching equivalence, there are a total of 902 graphs outside $\gen\single$ with smallest Seidel eigenvalue at least $-3$ with the following statistics, where the rank refers to the rank of $S(G) + 3I$ for a graph $G$.

    \begin{flushleft}
        \begin{tabular}{ccccccccccccc}
            $n$ & 5 & 6 & 7 & 8 & 9 & 10 & 11 & 12 & 13 & 14 & 15 & 16\\
            \hline
            \# & 2 & 5 & 12 & 20 & 35 & 52 & 69 & 89 & 101 & 103 & 101 & 90\\
            \hline
            minimum rank & 5 & 5 & 5 & 5 & 5 & 5 & 6 & 6 & 6 & 6 & 6 & 6
        \end{tabular}
    \end{flushleft}
    \begin{flushright}
        \begin{tabular}{cccccccccccc}
            17 & 18 & 19 & 20 & 21 & 22 & 23 & 24 & 25 & 26 & 27 & 28\\
            \hline
            70 & 54 & 37 & 23 & 16 & 10 & 5 & 3 & 2 & 1 & 1 & 1\\
            \hline
            7 & 7 & 7 & 7 & 7 & 7 & 7 & 7 & 7 & 7 & 7 & 7
        \end{tabular}
    \end{flushright}
    Moreover, up to switching equivalence, the graph achieving the minimum rank is unique for $n \in \sset{8,9,10,14,15,16,26,27,28}$, and when $n \in \sset{9,10,15,16,27,28}$, the unique graph is switching equivalent to $L(K_{3,3})$, the Petersen graph, $L(K_6)$, the Clebsch graph, the complement of the Schl\"afli graph, and the complement of $L(K_8)$, respectively. \hfill\computer
\end{theorem}

\cref{thm:main-1} is consistent with \cite[Theorem~3.1]{CKMY21} of Cao, Koolen, Munemasa and Yoshino and \cite[Corollary~1.1]{Y25} of Yoshino, which supports the correctness of our implementation.

In \cite{CKMY21}, Cao et al.\ classified, up to switching equivalence, graphs $G$ with largest Seidel eigenvalue $3$ that is maximal in the sense that there is no proper supergraph $G'$ of $G$ with largest Seidel eigenvalue $3$ such that $\rank\mleft(3I - S(G')\mright) = \rank\mleft(3I - S(G)\mright)$. The unique graphs with smallest Seidel eigenvalue $-3$ in \cref{thm:main-1} for $n \in \sset{10, 16, 28}$ are clearly maximal in the sense that there is no proper supergraph $G'$ of $G$ with smallest Seidel eigenvalue $-3$ such that $\rank\mleft(S(G') + 3I\mright) = \rank\mleft(S(G) + 3I\mright)$. Since taking the complement negates the Seidel matrix, the complements of these $3$ unique graphs are maximal in the sense of Cao et al. Indeed, the complement of the Petersen graph is $L(K_5)$; the complement of the Clebsch graph is switching equivalent to the vertex-disjoint union of $L(K_6)$ and $K_1$; the complement of the complement of $L(K_8)$ is just $L(K_8)$ --- all of these complement graphs appear in \cite[Theorem~3.1]{CKMY21}.

In \cite{Y25}, Yoshino derived the total number $s(n)$ of systems of $n$ equiangular lines with fixed angle $\arccos(1/3)$, which differs from the number of graphs in \cref{thm:main-1} by precisely $\floor{n/2} + 1$ for $n \ge 3$. The difference comes from the graphs in $\cl\single$ on $n$ vertices, which are not counted in \cref{thm:main-1}.

Our code is available as \texttt{min\_forb.cpp} and \texttt{enum.cpp}, as auxiliary files in the arXiv version of the paper. To run the code for $\alpha = 1/3$, use the additional \texttt{-{}-single-edge} flag. The output files containing the graphs in \cref{thm:main-1} together with their rank are available as \texttt{single-edge-5.txt} to \texttt{single-edge-28.txt}.

\subsection{Finiteness of exceptional graphs}

For general $\alpha \in (0,1)$, set $\lambda = (1-\alpha)/(2\alpha)$, and let $\G(\lambda)$ be the set of connected graphs with spectral radius at most $\lambda$. A general construction of graphs $G$ with smallest Seidel eigenvalue at least $-1/\alpha$ is to take the vertex-disjoint union of graphs from $\G(\lambda)$. Indeed, such a graph $G$ satisfies $\lambda I - A(G) \succeq 0$, which implies that $S(G) + I/\alpha = J + 2 (\lambda I - A(G)) \succeq 0$. Taking the switching closure yields the family $\gen{\G(\lambda)}$ whose members all have smallest Seidel eigenvalue at least $-1/\alpha$. This leads to the following qualitative question.

\begin{problem}
    For which $\alpha \in (0,1)$, are there only finitely many graphs outside $\gen{\G(\lambda)}$ with smallest Seidel eigenvalue at least $-1/\alpha$, where $\lambda = (1-\alpha)/(2\alpha)$?
\end{problem}

We note that the answer to the above problem is negative when $\alpha = 1/(4\sqrt{u^2+1}-1)$ for any integer $u \ge 10^5$. In this case, $\lambda = 2\sqrt{u^2+1}-1$ is not the spectral radius of any graph, and so for any graph $G \in \gen{\G(\lambda)}$ on $n$ vertices, the matrix $S(G) + I/\alpha$ is positive definite, and so its rank is $n$. If there were only finitely many graphs outside $\gen{\G(\lambda)}$ with smallest Seidel eigenvalue at least $-1/\alpha$, then we would conclude, by an argument similar to that in the proof of \cref{thm:main}, that $N_\alpha(d) = d$ for sufficiently large $d$. However, Schildkraut \cite[Theorem~1.4]{S23} proved that $N_\alpha(d) \ge d + \Omega(\log\log d)$.


\bibliographystyle{plain}
\bibliography{seidel}

\end{document}